\documentclass[11pt]{amsart}
\usepackage{hyperref}
\hypersetup{pdftex, colorlinks, citecolor=magenta, linkcolor=blue}
\usepackage{amsthm}
\usepackage{amssymb}
\usepackage{bbm}
\usepackage{eucal}
\usepackage{url}
\usepackage{dsfont}
\usepackage{upgreek}
\usepackage{mathrsfs}
\usepackage[all, arc]{xy}
\usepackage[perpage]{footmisc}
\usepackage[toc, page]{appendix}
\usepackage{enumerate}
\usepackage{lettrine}
\usepackage[lite, numeric,sorted,traditional-quotes,compressed-cites]{amsrefs}

\newcommand{\id}[1]{{\text{id}}_{#1}}
\newcommand{\Alt}[2]{{\rm Alt}^{#1}(#2)}

\newcommand{\Filt}[2]{{\rm Fil}^{#1}_{#2}}
\newcommand{\Grr}[3]{{\rm gr}^{#1}_{#2}K_0(#3)}
\newcommand{\grr}[3]{{\rm gr}^{#1}_{#2}{#3}}
\newcommand{\chern}[1]{{\rm ch}(#1)}

\newcommand{\CCong}[3]{{#1}={#2}\ ({\rm mod\ }{#3})}

\newcommand{\Z}{{\mathbf Z}}
\newcommand{\Q}{{\mathbf Q}}

\newcommand{\spec}[1]{{\text{Spec}(#1)}}

\newcommand{\kernel}{\text{ker}}
\newcommand{\cali}[1]{{\mathscr #1}}

\newcommand{\dime}[1]{{\rm dim}(#1)}

\newcommand{\class}[2]{{\rm cl}_{#1}({#2})}
\renewcommand{\id}[1]{{\rm id}_{#1}}

\newcommand{\proj}[3]{{\rm pr}^{#1}_{#2}(#3)}

{\newtheoremstyle{astatement}
{13pt}
{13pt}
{\it}
{}
{\bf}
{.$-$}
{.5em}
{}

\theoremstyle{astatement}

}
\newtheoremstyle{adefinition}
{13pt}
{13pt}
{}
{}
{\scshape}
{.}
{.5em}
{}

\theoremstyle{adefinition}

\newtheoremstyle{statement}
{13pt}
{13pt}
{\it}
{}
{\bf}
{.$-$}
{.5em}
{}

\theoremstyle{statement}

\newtheorem{theorem}{Theorem}[section]
\newtheorem{lemma}[theorem]{Lemma}
\newtheorem{proposition}[theorem]{Proposition}

\newtheorem{cor}[theorem]{Corollary}

\newtheorem{conjecture}[theorem]{Conjecture}

\newtheoremstyle{definition}
{13pt}
{13pt}
{}
{}
{\scshape}
{.}
{.5em}
{}

\theoremstyle{definition}

\newtheorem{example}[theorem]{Example}
\newtheorem{remark}[theorem]{Remark}

\newtheoremstyle{definitionprime}
{13pt}
{13pt}
{}
{}
{\scshape}
{$'$.}
{.5em}
{}

\theoremstyle{definitionprime}

\newtheoremstyle{remarks}
{13pt}
{13pt}
{}
{}
{\scshape}
{.}
{.5em}
{}

\theoremstyle{remarks}

\newtheoremstyle{remarksb}
{13pt}
{13pt}
{}
{}
{\scshape}
{.}
{\newline}
{}

\theoremstyle{remarksb}

\newtheoremstyle{underlined}
{13pt}
{13pt}
{\sl}
{}
{\scshape}
{}
{.5em}
{}

\theoremstyle{underlined}

\makeatletter 

\def\section@cntformat{\S\thesection.\ }
\def\subsection@seccntformat{\S\thesubsection \ }

\makeatother

\begin{document}
\title{A note on the Grothendieck group of an abelian variety}


\author{Shahram Biglari}
\address{School of Mathematics, Institute for Research in Fundamental Sciences,  P.~O.~Box 19395-5746, Tehran, IRAN}
\curraddr{} \email{biglari@ipm.ir}
\thanks{}



\keywords{Grothendieck ring, abelian variety, $K_0-$cycle, Fourier-Mukai transform, Beauville decomposition.}

\maketitle
\begin{abstract}
 We reformulate a conjecture of Beauville on algebraic cycles on an abelian variety in terms of certain compatibility and vanishings of some naturally defined filtrations on the Grothendieck group of the abelian variety.
\end{abstract}

\section{Introduction}\label{introduction}
In~\cite{Beauville_1986}, Beauville defines for any abelian variety $A$ over an algebraically cosed field $k$ of characteristic zero, the subgroup $CH^p_s(A)$ of the Chow group $CH^p(A)\otimes\Q$ for any codimension $p\geq 0$ and $s\in \Z$ to consist of elements $x$ such that
\[
 n_A^\ast (x)=n^{2p-s}x
\]
for all $n\in \Z$ where $n_A$ is the endomorphism $a\mapsto na$ on $A$. He has made the following:
\begin{conjecture}[Beauville]
 If $s<0$, then $CH^p_s(A)=0$.
\end{conjecture}
The original approach to study the conjecture has used the Fourier-Mukai transform. Here we follow a slightly different approach and give a new formulation of the conjecture.\\

The idea is to use the Grothendieck ring $K_0(A)$ of the variety $A$ as the basic setting for questions regarding cycles. By Riemann-Roch there is a homomorphism
\[
 K_0(A)\to CH^\bullet (A)\otimes\Q
\]
which is an isomorphism up to torsion. Using $K_0$ instead of $CH^\bullet$ and the clearer algebraic relations between its elements found through the $\lambda-$ring structure, makes $K_0$ a better setting for us to study questions related to the conjecture above.\\

The theory of Riemann-Roch starts with introducing the notion of a natural (Grothendieck) $\lambda-$ring structure on the ring $K_0(A)$ that is defined through the usual tensor product of complexes of bundles and the corresponding symmetric or alternating quotient bundles. Using the Pontryagin product instead, it is possible to define a $\lambda-$ring structure (with Adams operations denoted by $\psi^n_\star$) on the ring $K_0(A)$ equipped with the Pontryagin product $x\otimes y\mapsto x\star y$. Through a Chern theory this gives the Chow theory $CH_\bullet(A)$, i.e. the Chow group graded by integers determined by both $p,s$ (as before the statement of the conjecture above). We then use the machinery of Fourier-Mukai transform to show that certain careful change in the Adams operations $\psi^n_\star$ give a collection of maps $\psi_\pi^n$ satisfying the usual conditions of Adams operations. We thus  define a certain $\lambda-$ring structure on the ring $K_0(A)\otimes\Q$ with its usual product. As apposed to the usual $\lambda-$ring structure this new structure, naturally, sees the abelian structure of the variety $A$; for example if $L$ is a line bundle on $A$ and $n\geq 1$, then $\psi^n(L)=L^n$ whereas $\psi_\pi^n(L)=L$ for $L$ anti-symmetric and $\psi_\pi^n(L)={L}^n$ for $L$ symmetric. This new $\lambda-$ring structure and through an standard construction gives a filtration
\[
 K_0(A)\otimes\Q=\Filt{0}{\pi}\supseteq \Filt{1}{\pi}\supseteq\dots\supseteq \Filt{g+1}{\pi}=0.
\]
We call this the Pontryagin filtration. We can give the following:
\begin{conjecture}
 $\Filt{q}{\pi}\subseteq \Filt{q}{\gamma}$ for all $q\geq 0$.
\end{conjecture}
This is, in our opinion, a natural reformulation of the conjecture. This paper is organized as follows.\\

After fixing some notations, we give an exposition on Fourier-Mukai transform on $K_0$. Immediately after this definition we consider the family of endomorphisms $m_A^\ast$ of $K_0(A)$ for $m\in \Z$. We prove this result: the sub-algebra of endomorphisms of $K_0(A)$ generated by all $m_A^\ast$ is as an abelian group free on $m_A^\ast$ for $m=0,1,\dots,2g=2\dime{A}$. We also give the universal relations between $m_A^\ast$ in an explicit formula. Next we consider the group $K_0(A)$ with its Pontryagin product and $\lambda_\star-$structure. This is shown to behave naturally and admit a Chern character theory. Finally we define the Pontryagin filtration on $K_0(A)$ (with its usual ring structure) and state the reformulation above.

In a last section, we show that it is also possible to define a certain ``composition'' of the two $\lambda-$ring structures above (a general notion of compositions of $\lambda-$ring structures is completely absent from the literature despite being such a natural notion in the category of $\lambda-$rings). This leads to a rather complicated $\lambda-$ring structure $\Gamma$. A second reformulation of Beauville conjecture is given in the form of a vanishing conjecture of higher filtrations of this composition structure. If this reformulated conjecture is true, it will furthermore show that each group $CH^p_s(A)$ is, according to a Riemann-Roch theorem, exactly an associated graded piece of certain filtration on $K_0(A)\otimes\Q$. 

The methods in this note will be used in future to establish certain formulas in the Grothendieck group of the Jacobian of a smooth projective curve using the language of $\lambda-$rings.\\

\emph{Acknowledgment}. I would like to thank Prof. A. Beauville for carefully reading an earlier version of this paper and finding a gap in it.\\

In what follows the base field $k$ is assumed to be algebraically closed and of characteristic zero.
\section{Notations and preliminaries}
The theory of Grothendieck group of schemes is presented in [SGA 6, Exp. IV, \S 2]: for a noetherian (separated) scheme $X$ let $K_0'(X)$ be the Grothendieck group of the derived category of (cohomologically) bounded complexes of ${\mathscr O}_X-$modules with coherent cohomology modules. This is an abelian group and can also be defined as the (\emph{na\"{i}ve}) Grothendieck group of the abelian category of coherent modules. For smooth schemes $S$ this is also the same as the Grothendieck group $K_0(S)$ of the exact category of locally free coherent $\cali{O}_S-$modules. These constructions are functorial. More precisely, $S\mapsto K_0(S)$ defines a contravariant functor from schemes to commutative rings and for smooth schemes $K_0'$ defines a covariant functor for proper morphisms.\\

We shall also use the projection formula in various settings. The formula is given in [SGA 6, Exp. IV, 2.11-12].\\

Recall the definition of coniveau (or topological) filtration on $K_0'(X)$: for each $j\geq 0$, the subgroup ${\rm Fil}_{\rm top}^j\subseteq K_0'(X)$ is defined to be generated by classes of coherenet modules having a support of codimention $\geq j$. Note that by definition
\[
\Filt{i}{\rm top}=0\quad{\rm for\ all\ }i>\dime{X}.
\]
The filtration is functorial in a suitable sense and in cases we deal with. This and other results are explained in [SGA 6, Exp. X]. We denote by $\Grr{\bullet}{\rm top}{S}$ (or $\Grr{\bullet}{}{S}$) the graded group associated to the filtration above.\\

There is a $\lambda-$ring structure on $K_0(S)$ defined by $\lambda^n(x)=\class{}{\Alt{n}{E}}$ where $x=\class{}{E}$ for a locally free coherent sheaf $E$ on $S$. The rank function $\epsilon(x)={\rm rk}(E)$ gives an augmentation. There is therefore defined the gamma (or Grothendieck) filtration $\Filt{n}{\gamma}$ on the ring $K_0(S)$; $\Filt{0}{\gamma}=K_0(S)$, $\Filt{1}{\gamma}={\rm ker}(\epsilon)$ and $\Filt{n}{\gamma}$ is the subgroup generated by all products $$\gamma^i(x)\cdot\gamma^i(x)\cdot\ldots\cdot \gamma^k(z)$$ with $i+j+\dotsb +k\geq n$ and $x,y,\dotsc,z\in \Filt{1}{\gamma}$. The associated graded group is denoted by $\grr{\bullet}{\gamma}{K_0(S)}$. The group $K_0'(S)$ becomes a filtered module over the filtered ring $K_0(S)$. There is a natural map
\[
 \rho\colon \grr{\bullet}{\gamma}{K_0(S)}\to \Grr{\bullet}{\rm top}{S}
\]
which is an isomorphism over the rational numbers.
\section{The Fourier-Mukai theory on $K_0$}
Let $A\to \spec{k}$ be an abelian variety of dimension $g$ over $\spec{k}$. Denote by $0_A$ the identity of $A$ as a group over $k$. Let $\hat{A}$ be the dual abelian variety of $A$ and $P$ the Poincar\'e line bundle on $A\times_k\hat{A}$; the unique (up to isomorphism) line bundle on $A\times_k\hat{A}$ whose inverse image along the morphisms $\hat{A}\to A\times\hat{A}$ (given by $\hat{x}\mapsto 0_A\times {\hat{x}}$) is trivial and for all (closed points) $\alpha\in \hat{A}$ the restriction $P_\alpha$ of $P$ to $A\times {\alpha}$ is represented by $\alpha$ under the natural isomorphism $\hat{A}\simeq {\rm Pic}^0(A)$ (see~\cite[II, 8.]{mumford-1985}). Let $p$ be the $K_0-$class of $P$. This defines the so-called Fourier-Mukai transform
\[
F_p\colon K_0(A)\to K_0(\hat{A}),\quad x\mapsto \proj{}{\hat{A}_\ast}{\proj{\ast}{A}{x}\cdot p}.
\]
Similarly the Poincar\'e line bundle $\hat{P}$ on ${B}\times \hat{B}$ where $B=\hat{A}$ defines a morphism $F_{q}\colon K_0(\hat{A})\to K_0({A})=K_0(\hat{B})$ where $q=\class{}{\hat{P}}$. More generally for any element in $u\in K_0(A\times \hat{A})$ the same formula above (with $p$ replaced by $u$) defines a homomorphism of abelian groups;
\[
 K_0(A\times\hat{A})\to {\mathrm{Hom}}_{\Z} (K_0(A), K_0(\hat{A})),\quad u\mapsto u^\ast=F_u
\]
that is compatible with compositions of $K_0-$correspondences. For each integer $m$ let the endomorphism $(m_A)^\ast$ (resp. $(m_A)_\ast$) of $K_0(A)$ be the $K_0$ (resp. $K_0'$) of the morphism $m_A\colon A\to A$ given by $x\mapsto mx$. Note that
\[
 0_A^\ast(x)={\mathrm {rk}}(x),\quad (0_A)_\ast(x)=\chi (x)[0_A]
\]
for all $x\in K_0(A)$. The result \cite[2.2]{mukai_1981} shows that $F_q\circ F_{p}$ is exactly the homomorphism $(-1)^g(-1_A)^\ast$ by showing that the $K_0-$correspondence $q\circ p$ in $K_0(A\times \hat{A})$ is represented by the complex
\[
(\text{the graph of $-1_A\colon A\to A$})_{\ast} (\cali{O}_A)[-g]
\]
in the derived category $D^b_{\mathrm{coh}}(A\times \hat{A})$. We may generalize this:
\begin{proposition}\label{F_qmF_pn} 
For any integers $n,m$, we have
\[
 F_{q^m}\circ F_{p^n}=(-1)^g(-m_A)^\ast\circ (n_A)_\ast
\]
\end{proposition}
\begin{proof}
We note that with notations as above, there is an isomorphism $P^{\otimes n}\simeq (n_A\times \id{\hat{A}})^\ast P$ and similarly $(\id{\hat{A}}\times m_{{A}})^\ast \hat{P}\simeq \hat{P}^{\otimes m}$. Taking $K_0-$classes of these and using the usual arguments (i.e. the projection formula and the base change), it follows from the compatibility of the Fourier-Mukai transforms with the compositions of correspondences that 
\begin{equation*}
\begin{split}
 F_{q^m}\circ F_{p^n}(a) & = F_{(\id{\hat{A}}\times m_{{A}})^\ast q}\circ F_{(n_A\times \id{\hat{A}})^\ast p}(a)\\
 & = m_{{A}}^\ast \circ F_q \circ F_p \circ {n_{A}}_\ast (a)\\
 & = (-1)^gm_A^\ast \circ (-1_A)^\ast \circ (n_A)_\ast (a)\\
 & = (-1)^g (-m_A)^\ast \circ (n_A)_\ast (a).
\end{split}\end{equation*}
The result follows.
\end{proof}
We recall the statement of Grothendieck-Riemann-Roch in this setting. Using the same principle as above and denoting by $p$ the class of the Poincar\'e bundle in $\grr{\bullet}{\gamma}{K_0(A\times\hat{A})}$, as above there is a homomorphism
\[
 F^{\rm gr}_p:\colon \Grr{\bullet}{\gamma}{{A}}\otimes\Q\to \Grr{\bullet}{\gamma}{\hat{A}}\otimes\Q.
\]
Let us denote the map $K_0(A)\to \Grr{\bullet}{\gamma}{A}\otimes\Q$ appearing in the Riemann-Roch theorem of Grothendieck by ${\rm ch}={\rm ch}_A$ (the correct notation is ${\rm ch}\circ \tilde{c}$). For example if $L$ is a line bundle on $A$ and $x=\class{}{L}$, then the Chern class $c^j(x)=0$ for all $j>1$, $\CCong{c^1(x)}{x-1}{\Filt{2}{}}$ and
\[
 \chern{x}={\rm exp}( c^1(x)) =1+c^1(x)+\frac{1}{2!}c^1(x)^2+\cdots .
\]

\begin{theorem}\label{GRR}
The diagram 
\[
 \xymatrix{
 K_0(A)\ar[d]_-{{\rm ch}_A} \ar[r]^-{F_p} & K_0(\hat{A})\ar[d]^-{{\rm ch}_{\hat{A}}}\\
 \Grr{\bullet}{\gamma}{{A}}\otimes\Q\ar[r]^-{F^{\rm gr}_p} & \Grr{\bullet}{\gamma}{\hat{A}}\otimes\Q
 }
\]
is commutative. Moreover ${F^{\rm gr}_p}$ is an isomorphism with inverse $(-1)^g{F^{\rm gr}_{q^{-1}}}$.
\end{theorem}
\begin{proof}
 This follows from Grothendieck-Riemann-Roch for abelian varieties (where the relative Todd class of a homomorphism of abelian varieties is trivial).
\end{proof}
The above can be used to prove the following:
\begin{proposition}\label{m_* independence}
The subalgebra of endomorphisms of $K_0(A)$ generated by $(m_A)_\ast$ for $m\in \Z$ is as an abelian group freely generated by $(m_A)_\ast$ for $m=0,1,\dots,2g$.
\end{proposition}
\begin{proof}
First note that $1-p\in \Filt{1}{\gamma}$. Using~\ref{F_qmF_pn} and the equation $p=\sum_{k=0}^{2g}(1-p^{-1})^{k}$
in $K_0(A\times \hat{A})$ we find that
\[
 \begin{split}
  \id{}& =(-1)^gF_{q^{-1}}\circ F_p\\
  & = (-1)^g \sum_{k=0}^{2g} F_{q^{-1}}\circ F_{(1-p^{-1})^k}\\
  & = (-1)^g \sum_{k=0}^{2g} \sum_{m=0}^{k} (-1)^m\binom{k}{m}F_{q^{-1}}\circ F_{p^{-m}}\\
  & = \sum_{m=0}^{2g} (-1)^m\binom{2g+1}{m+1} (-m_A)_\ast.
 \end{split}
\]
Next we show that if $k$ is an integer with $k>2g$ or $k<0$, then the endomorphism $(k_A)_\ast$ of $K_0(A)$ is a $\Z-$linear combination of $(m_A)_\ast$ for $m=0,1,\dots,2g$. Let $k>2g$. Since $(1-p)^k=0$, it follows from the same argument as above that a sum of $(j_A)_\ast$ for $j=0,1,\dots,k$ vanishes. The assertion for $k>2g$ follows from this by induction. To prove the assertion for $k<0$, we need only to consider the case $k=-1$. This case follows by applying $(-1)_\ast$ to the first equation in this proof. Now we show that $(m_A)_\ast$ for $m=0,1,\dots,2g$ are $\Z-$linearly independent. Assume otherwise: i.e. there is a relation
\[
\Lambda:= \sum_{m=0}^{2g}\lambda_m (m_A)_\ast=0
\]
as endomorphisms of $K_0(A)$ with $\lambda_m\in \Z$. Since each $(m_A)_\ast$ preserves the gamma filtration (up to torsion), by~\ref{GRR} the same relation as $\Lambda$ holds for $(m_A)_\ast$'s as endomorphisms of $\Grr{\bullet}{\gamma}{A}\otimes\Q\simeq CH^\bullet (A)\otimes\Q$. Therefore we need only to show that
\[
\Lambda\colon \Grr{\bullet}{\gamma}{A}\otimes\Q\to \Grr{\bullet}{\gamma}{A}\otimes\Q
\]
is zero if and only if $\lambda_m=0$ for all $m\leq 2g$. By~\cite[Prop. 4]{Beauville_1986}, we know that for all $0\leq j\leq g$ and $g-j\leq i\leq g$ there is an element $0\neq x\in \Grr{j}{\gamma}{A}\otimes\Q$ with $(m_A)_\ast(x)=m^{g-j+i}x$ for all $m\in \Z$. The equation $\Lambda (x)=0$ implies that $\sum \lambda_m m^{g-j+i}=0$. Using this for all values of $i,j$ as above, we conclude that 
\[
 \sum \lambda_m m^{k}=0,\text{ for all $0\leq k\leq 2g$}.
\]
Note that the $(2g+1)\times (2g+1)-$matrix $X$ with the entry $m^{k}$ in position $(m,k)$ where $0\le m,k\leq 2g$ is invertible. It follows that $\lambda_m=0$ for all $m$. This completes the proof.
\end{proof}
\begin{remark}
 The universal relation between the endomorphism $(m_A)_\ast$ is given by the formula in the proof of~\ref{m_* independence} expressing $(-1_A)_\ast$ in terms of $(m_A)_\ast$ for $m=0,1,\dots,2g$. Also note that for $m\neq 0$, the homomorphism $(m_A)_\ast$ is an automorphism on $K_0(A)\otimes\Q$. There are many other classes of automorphisms of $K_0(A)\otimes\Q$ (see below~\ref{prop:omega_n}). 
\end{remark}

\begin{remark}
 Almost the same proof as in~\ref{m_* independence} shows that a similar statement for the sub-algebra generated by $m_A^\ast$ for $m\in \Z$ holds.
\end{remark}

\section{Pontryagin $\lambda_\star-$structure}
We shall introduce a filtration on $K:=K_0(A)$ induced from a natural $\lambda-$ring structure on $K_0(A)$ arising from the Pontryagin product.\\

Let $\pi\colon A\to \spec{k}$ be an abelian variety of dimension $g$ over a field $k$. The Pontryagin product on $K_0(A)$ is defined as follows;
\[
\star\colon K_0'(A)\otimes K_0'(A)\to K_0'(A),\quad x\otimes y\mapsto m_\ast (\proj{\ast}{1}{x}\cdot \proj{\ast}{2}{y})
\]
where $m\colon A\times_kA\to A$ is the multiplication morphism (which is projective) and the projections are from $A\times_kA$. The unit of $A$ as a group scheme is denoted by $e_A\colon \spec{k}\to A$. This is a closed immersion.
\begin{lemma}
The Pontryagin product endows $K_0(A)$ with the structure of a commutative associative ring with unit $e:=e_{A,\ast}(1_S)=:[0_A]$.
\end{lemma}
\begin{proof}
We only show that $[0_A]$ is the unit element; for this let $x\in K_0(A)$ and note that
\begin{alignat*}{2}
x\star 0_A & =m_\ast (\proj{\ast}{1}{x}\cdot \proj{\ast}{2}{[0_A]})\\
& =m_\ast(\proj{\ast}{1}{x}\cdot (\id{}\times e)_\ast (1))&\quad\text{by base change}\\ 
& =m_\ast\circ (\id{}\times e)_\ast\bigl((\id{}\times e)^\ast\circ \proj{\ast}{1}{x}\cdot 1\bigr)&\quad\text{by projection formula}\\
& =x. 
\end{alignat*}
Similarly it follows that $0_A$ is a right unit.
\end{proof}
\begin{remark}\label{rem:f_* and pontryagin product}
Let $f\colon A\to B$ be a homomorphism of abelian varieties over $\spec{k}$. It is not difficult to show that
\[
f_\ast (x\star y)=f_\ast (x)\star f_\ast (y).
\]
Moreover $f_\ast (0_A)=0_B$. In other words, $f_\ast$ is a ring homomorphism with respect to the Pontryagin product.
\end{remark}
Using the rank function $\epsilon:={\rm rk}_A\colon K_0(A)\to \Z$, the ring $K_0(A)$ (with the usual product) becomes an augmented $\Z-$algebra. Similarly, the ring $K_\star=(K_0(A),\star)$ is also an augmented $\Z-$algebra: let $\chi=\chi_A\colon K_0(A)\to \Z$ be the Euler-Poincar\'e characteristic, i.e. $\chi_A(x)=\pi_\ast (x)$ and note that:
\begin{lemma}
For any $x,y\in K_0(A)$ we have
\[
 \chi_A(x\star y)=\chi_A(x)\cdot \chi_A(y).
\]
\end{lemma}
\begin{proof}
This follows from~\ref{rem:f_* and pontryagin product} and the fact that on $\Z=K_0(\spec{k})$ the Pontryagin product agrees with the usual product.
\end{proof}

Let $n\geq 0$ and $x\in K_0(A)$. Define
\[
 \lambda_\star^n(x)=F_p^{-1}\circ \lambda^n_{\hat{A}}\circ F_p(x).
\]
These define a $\lambda-$ring structure on $(K_0(A),\star)$ which we call the Pontryagin $\lambda_\star-$structure.
\begin{remark}
 It is possible to define this structure using the Pontryagin product on the level of derived category of coherent modules on $A$. For our purpose here we have chosen the above equivalent definition.
\end{remark}

\begin{theorem}\label{thm:fm-iso}
The Fourier-Mukai transform $F_p\colon (K_0(A),\star)\to K_0(\hat{A})$ is an isomorphism of $\Z-$augmented $\lambda-$rings.
\end{theorem}
\begin{proof}
The homomorphism $F_p$ is by~\ref{F_qmF_pn} an isomorphism of abelian groups. Moreover $F_p$ commutes with the augmentations; i.e.
\[
 \chi_A(x)={\rm rk}_{\hat{A}}(F_px)
\]
which follows from the definitions and usual arguments (i.e. using base change and projection formula). Finally let $x,y\in K_0(A)$. It follows from the usual arguments that
\begin{equation*}
F_K(x\star y) ={\rm pr}_{3,\ast}\bigl(\proj{\ast}{1}{x}\cdot \proj{\ast}{2}{y}\cdot (m\times\id{\hat{A}})^\ast p\bigr)
\end{equation*}
where the projections are from $A\times_kA\times_k\hat{A}$. Note that by the theorem of the cube (\cite[II, \S 6]{mumford-1985}) there is an isomorphism
\[
(m\times\id{\hat{A}})^\ast (P)\simeq \proj{\ast}{13}{P}\otimes \proj{\ast}{23}{P}
\]
where $P$ is the Poincar\'e bundle on $A\times_k\hat{A}$. From this it follows that $F_p(x\star y)=F_p(x)\cdot  F_p(y)$. Therefore $F_p$ is an isomorphism of augmented algebras. Compatibility of $F_p$ with the $\lambda-$ring structures is evident.
\end{proof}
As usual define the power series
\[
 \lambda_\star (x)=\lambda_\star (x,t)=\sum_{n\geqq 0} \lambda_\star^n(x)t^n.
\]
From this the $\gamma-$ring structure of $K_\star$ is constructed; i.e. as series $$\gamma_\ast (x):=\lambda_\star (x,t/{1-t}).$$Finally the corresponding filtration is denoted by $\Filt{\bullet}{\star}$. Note that $\Filt{0}{\star}=K_\star$, $\Filt{1}{\star}=\kernel (\chi)=:I$ and in general $\Filt{q}{\star}$ is the subgroup generated by elements of the form
\[
 \gamma_\star^{i}(x)\star \gamma_\star^{j}(y)\star\dots\star \gamma_\star^{k}(z)
\]
with $i,j,\dots, k\geq 0$, $i+j+\dots +k\geq q$ and $x,y,\dots,z\in \Filt{1}{\star}$. It follows from the definition that
\[
 \Filt{n+1}{\star}\subseteq \Filt{n}{\star}.
\]
We consider this filtration on $K_\star=(K,\star)$ or on the additive group of $K=K_0(A)$. We call this the (Pontryagin) $\star-$filtration. Denote the associated quotient groups of this filtration by \[\grr{n}{\star}{K}:=\Filt{n}{\star}/{\Filt{n+1}{\star}}.\]
\begin{cor}
 $\Filt{n}{\star}=0$ for all $n\geq g+1$.
\end{cor}
\begin{proof}
 By~\ref{thm:fm-iso} we know that $F_p\Filt{n}{\star}\subseteq \Filt{n}{\gamma}$. Since $F_p$ is an isomorphism and
 \[
  \Filt{n}{\gamma}\subseteq \Filt{n}{\rm top}
 \]
we conclude that $\Filt{n}{\star}=0$ for $n>{\rm dim}(A)=g$.
\end{proof}
\begin{remark}
 Also note that by definition $I^{\star n}\subseteq \Filt{n}{\star}$.
\end{remark}
\begin{cor}
$(m_A)_\ast\circ \gamma^n_\star=\gamma^n_\star\circ (m_A)_\ast$ for any $m\in \Z$ and $n\geq 0$.
\end{cor}
\begin{proof}
Note that $(m_{\hat{A}})^\ast$ is a $\lambda-$morphism of $\widehat{K}$. The result follows from~\ref{thm:fm-iso}.
\end{proof}
Using~\ref{thm:fm-iso} we obtain a well-defined homomorphism
\[
 \grr{}{}{F_p}\colon \grr{\bullet}{\star}{A}\to \grr{\bullet}{\gamma}{\hat{A}}
\]
which is graded of degree $0$. From the general theory of $\lambda-$rings and the Chern character in [SGA 6, Exp. V] we know that there is a theory of Chern charcters ${\rm ch}_\star\colon K_\star\to \grr{\bullet}{\star}{K}\otimes\Q$ and for which there is a commutative diagram of isomorphisms of $\Q-$algebras:
\[
 \xymatrix{
 K_\star\otimes\Q\ar[r]^-{F_p} \ar[d]^-{{\rm ch}_\star} & \widehat{K}\otimes\Q\ar[d]^-{{\rm ch}} \\ \grr{\bullet}{\star}{K}\otimes\Q\ar[r]^-{\grr{}{}{F_p}} & \grr{\bullet}{\gamma}{\widehat{K}}\otimes\Q.
 }
\]
We may reformulate this relationship between the two Chern character: let
$$\rho_K\colon ={F^{\rm gr}_{{p}}}^{-1}\circ \grr{}{}{F_p}\colon \grr{\bullet}{\star}{K}\otimes\Q\to \grr{\bullet}{\gamma}{K}\otimes\Q$$
Using the functoriality of the Chern character, the diagram above and the Grothendieck-Riemann-Roch we obtain that the diagram
\[
  \xymatrix{
 K_0(A)\otimes\Q\ar[d]^-{{\rm ch}_\star} \ar[drr]^-{{\rm ch}}\\
 \grr{\bullet}{\star}{K}\otimes\Q\ar[rr]^-{\rho_K} && \grr{\bullet}{\gamma}{K}\otimes\Q
 }
\]
of vector spaces and linear maps is commutative. Moreover for any $x,y\in \grr{\bullet}{\star}{K}$ we have
\[
 \rho_K(x\star y)=\rho_K(x)\star \rho_K(y)
\]
where Pontryagin product on $\grr{\bullet}{\gamma}{K}\otimes\Q$ is (well-)defined similar to the case of the Grothendieck ring. We also have:
\begin{cor}\label{lem:ch-star}
For any $x,y\in K$, we have $\chern{x\star y}=\chern{x}\star \chern{y}$.
\end{cor}
By definition ${\rho}_K([0_A])=[0_A]$. Putting these together, we conclude that $\rho_K$ is a $\Q-$algebra (non-graded) isomorphism.

\section{Comparison of two Gamma filtrations}

We let $\psi^n_\star$ for $n\geq 1$ be the Adams operations corresponding to the Pontryagin $\lambda_\star-$structure on $K_\star=(K_0(A),\star)$. Similarly $\psi^n$ for $n\geq 1$ denote the Adams operations on $K$ corresponding to the usual (Grothendieck) $\lambda-$ring structure. 

We denote by $K^{(n)}\subseteq K_0(A)\otimes\Q$ the eigenspace of $\psi^\alpha$'s corresponding to the eigenvalue $\alpha^n$ and $K^{(n)}_\star=:K_{(n)}$ that of $\psi_\star^\alpha$ on $K_\star\otimes\Q$. Therefore
\begin{equation*}\label{f_pk_n}
 F_p K_{(n)}=K_{\hat{A}}^{(n)}=:\widehat{K}^{(n)}.
\end{equation*}
\begin{lemma}\label{lem:f_pkn}
 $F_p {K}^{(n)}=\widehat{K}_{(n)}$ for any integer $n\geq 0$.
\end{lemma}
\begin{proof}
This is equivalent to $F_q\circ F_p {K}^{(n)}=F_q\widehat{K}_{(n)}$. That is if ${K}^{(n)}=F_q\widehat{K}_{(n)}$ which is the equation appearing just before the assertion applied for $\widehat{A}$.
\end{proof}
\begin{proposition}\label{prop:eigenspace-decomp}
There is a direct sum decomposition
\[
 K_0(A)\otimes\Q=K^{(0)}\oplus K^{(1)}\oplus\dots\oplus K^{(g)}.
\]
\end{proposition}
\begin{proof}
It is not difficult to see that $x\in K^{(q)}$ if and only if
\[
 \chern{x}\in \grr{q}{\gamma}{K}\otimes\Q.
\]
Since the Chern character is an isomorphism, the result follows.
\end{proof}
\begin{remark}\label{rem:eigenspace-decomp}
Similarly, it follows from the results on ${\rm ch}_\star$ that $K_0(A)\otimes\Q$ is a direct sum of $K_{(q)}$'s.
\end{remark}
\begin{lemma}\label{lem:beauville-1986-prop-1}
For integers $m,n\geq 0$ and an element $x\in K^{(n)}$, the following assertions are equivalent:
\begin{enumerate}
 \item $x\in K^{(n)}\cap K_{(m)}$,
 \item $k_A^\ast (x)=k^{g+n-m}x$ for any integr $k\in \Z$.
\end{enumerate}
\end{lemma}
\begin{proof}
This can be proved using the same arguments as those in the proof of~\cite[Prop. 1]{Beauville_1986}. Details are left to the reader.
\end{proof}
\begin{lemma}\label{lem:eigenspace-decomp-2}
 For any integer $q\geq 0$, there is a direct sum decomposition
 \[
  K^{(q)}=K^q_0\oplus K^q_1\oplus\dots\oplus K^q_g
 \]
where $K^q_j:=K_{(j)}\cap K^{(q)}$ for $j\geq 0$.
\end{lemma}
\begin{proof}
Let $x_j\in K^q_j$ for $j=0,1,\dots,g$ be a family of elements of $K$ with $\sum x_j=0$. Applying $\psi_\star^n$ to this we obtain $\sum n^jx_j=0$. As this holds for all $n\geq 1$, we obtain $x_j=0$ for all $j$. Now let $x\in K^{(q)}$. Set $y=F_p(x)$. Using~\ref{prop:eigenspace-decomp} we may write
\[
 y=y_0+y_1+\dots+y_g,\quad y_j\in \widehat{K}^{(j)}.
\]
Set $x_j=F_p^{-1}(y_j)$. Therefore $x_j\in K_{(j)}$ and $x=\sum x_j$. Also note that by (a variant of)~\ref{lem:beauville-1986-prop-1} it follows that $x_j\in K^{(q)}$. The result follows.
\end{proof}
For a fixed integer $n\geq 1$, consider the abelian group homomorphism
\[
 \psi^n_\star\colon K_0(A)\otimes\Q\to K_0(A)\otimes\Q.
\]
\begin{proposition}\label{prop:omega_n}
The map
\[
 \pi_\star^n:=\psi_\star^n\otimes n^{-g}\id{}\colon K_0(A)\otimes\Q\to K_0(A)\otimes\Q
\]
is a morphism of $\Q-$augmented $\lambda-$rings.
\end{proposition}
\begin{proof}
It is clear that $\pi_\star^n$ is an endomorphism of the additive group of $K_0(A)\otimes\Q$. We show that $\pi_\star^n$ is a $\lambda-$morphism. For this we proceed as follows. We shall remove $\Q$ from the notations and work exclusively with rational coefficients. Therefore $K$ will denote $K_0(A)\otimes\Q$. The eigenspace decompositions of the Adams operations on $K_\star$ gives the decomposition
\[
 K_\star=K_\star^{(0)}\oplus K_\star^{(1)}\oplus\dots\oplus K^{(g)}_\star
\]
where $x\in K_\star^{(q)}$ if and only if $\psi_\star^n(x)=n^qx$ for $n\geq 1$. Also note that $\pi_\star^n(1)=1$. Therefore to show that $n^{-g}\psi_\star^n$ is a ring homomorphism, it is enough to note that for $x\in K_\star^{(\alpha)}$ and $y\in K_\star^{(\beta)}$, using~\ref{lem:ch-star}, we have
\[
x\cdot y\in K_\star^{(\alpha)}\cdot K_\star^{(\beta)}\subseteq K_\star^{(\alpha+\beta-g)}.
\]
Next we show that $n^{-g}\psi^n_\star$ commutes with each $\lambda^\alpha$ or equivalently with each $\psi^\alpha$. But the equality $\psi^n_\star\circ \psi^\alpha=\psi^\alpha\circ\psi^n_\star$ holds on each $K^{(r)}\cap K_{(s)}$ and hence on $K$. Finally we show that for each $x\in K$ we have
\[
 \epsilon (n^{-g}\psi_\star^n(x))=\epsilon (x).
\]
This is linear in $x$. Let $x\in K_{(q)}$ for an integer $0\leq q\leq g$. It follows that $\epsilon (n^{-g}\psi_\star^n(x))=n^{q-g}\epsilon (x)$. Now we may write $x=x_0+x_1+\dots+x_g$ with $x_j\in K^{(j)}$. Then $\epsilon (x)=\epsilon (x_0)$. Therefor $\epsilon(x)=0$ unless $q=g$. The result follows. 
\end{proof}
\begin{remark}\label{rem:omega_n}
Note that $\psi^n_\star\otimes\id{\Q}$ being an Adams operation is an isomorphism. 
\end{remark}

Let $n\geq 1$. Define $\psi^n_\pi(x)=n^{g-q}x$ for $x\in K_{(q)}$ and extend this by linearity for all $x\in K$ (note that $\psi^n_\pi$ is the inverse of $\pi_\star^n$). These are easily seen to define a $\lambda-$ring structure $\pi$ on $K=K_0(A)\otimes\Q$ for which the Adams operations are given by $\psi^n_\pi$. Let $\epsilon_\pi$ be the projection map $K\to K_{(g)}$. Note that $K_{(g)}$ is a $\Q-$subalgebra of $K$ and $\epsilon_\pi$ is a $\Q-$algebra homomorphism. Considering $K_{(g)}$ with its canonical binomial $\lambda-$ring structure, the map $\epsilon_\pi$ makes $K$ a $K_{(g)}$-augmented $\lambda-$ring whose $\lambda$, $\gamma$ operations are denoted by $\lambda^i_\pi$, $\gamma^i_\pi$. 

Let $x\in K_{(q)}$ with $q<g$. Using the definition of $\lambda_\pi$ (and $\gamma_\pi$) in terms of the Adams operations $\psi_\pi^n$, i.e.
\[
 \lambda_\pi (x,t)=\exp (\sum_{n\geq 1}\frac{(-1)^{n-1}}{n}\psi_\pi^n(x)t^n)
\]
and the equation $\gamma_\pi (x,t)=\lambda_\pi (x,t/{1-t})$, it follows that for all $i\geq 1$ we have the equation
\begin{equation}\label{eq:gamma}
\gamma^i_\pi(x)=a(i;g-q,1)x+a(i;g-q,2)x^2+\dots+a(i;g-q,g)x^g
\end{equation}
where $a(i;g-q,m)$ are universally defined rational numbers depending only on $i,g-q,m$. Note that $x^i=0$ if $iq-(i-1)g<0$. 
\begin{example}
It is easy to see that
\[
 a(i;g-q,1)=(-1)^{i-1}(i-1)!\begin{Bmatrix}g-q\\i\end{Bmatrix}.
\]
where the last number given in braces is a Stirling number of the second kind. The other values of $a(i;g-q,m)$ may be considered as higher variants of these numbers.
\end{example}
Using the gamma structure $\gamma_\pi$ and the augmentation $\epsilon_\pi$, we can as usual (cf.~\ref{thm:fm-iso}f.) define  a filtration of $K$ by subgroups denoted by $\Filt{n}{\pi}$. Note that $\Filt{1}{\pi}={\rm ker}(\epsilon_\pi)$ and the filtration is in fact $K_{(g)}=\grr{0}{\pi}{K_0(A)\otimes\Q}-$linear. We may call this filtration the Pontryagin filtration on $K_0(A)\otimes\Q$.
\begin{lemma}
 $\Filt{n}{\pi}=0$ for $n>g$.
\end{lemma}
\begin{proof}
We first show that if $x\in K_{(q)}$ with $0\leq q<g$, then we have $\gamma_\pi^{n}(x)=0$ for all $n\geq g+1$. The equation~\eqref{eq:gamma} is valid in any $\lambda-$ring (for appropriate $x$). Let $z\in K^{(g-q)}$ with $z^i\neq 0$ for all $i(g-q)\leq g$; a suitable power of $K^{(1)}-$component (i.e. in the decomposition~\ref{prop:eigenspace-decomp}) of any symmetric line bundle works. Since $\gamma^{n}(z)=0$ for all $n>g$, it follows from the uniqueness of the decomposition~\ref{prop:eigenspace-decomp} that $a(n;g-q,i)=0$ for all $i,n$ with $i(g-q)\leq g$ and $n>g$. Next we show that $\Filt{n}{\pi}=0$ for $n>g^2$. For this let $0\neq \zeta\in \Filt{n}{\pi}$. It follows that there are $x,y,\dots,z$ in ${\rm ker}(\epsilon_\pi)$ and integers $i,j,\dots,k\geq 0$ with $i+j+\dots+k\geq n$ and
\[
 \zeta'=\gamma^{i}_\pi(x)\cdot \gamma^{j}_\pi(y)\dots \gamma^{k}_\pi(z)\neq 0.
\]
We may assume that $x,y,\dots,z$ are homogeneous for $\pi$, i.e. $x\in K_{q(x)}$ for some $0\leq q(x)<g$ and similarly for $y,\dots,z$.
Since $\Filt{g+1}{\gamma}=0$, $\Filt{1}{\pi}\subseteq \Filt{1}{\gamma}$ and $n>g^2$, it follows that at least one of $i,j,\dots,k$, say $i$, is $\geq g+1$. Therefore $\zeta'=0$ from the first step. This contradicts the assumption and hence $\Filt{n}{\pi}=0$ for $n>g^2$. Finally we show that $\Filt{n}{\pi}=0$ for $n>g$. For this note that by the second step and the eigenspace decomposition of $K$ with respect to $\psi_\pi^n$'s (from the general theory of Chern character of $\lambda-$rings with discrete $\gamma-$filtration) we see that if $\Filt{q}{\pi}\neq 0$ for some $q>g$, then there is an element $x\neq 0$ with $\psi_\pi^n(x)=n^qx$ for all $n\geq 1$. But this contradicts the decomposition~\ref{lem:eigenspace-decomp-2} (or equivalently~\ref{prop:eigenspace-decomp} applied to $K_\star$). The lemma is proved.
\end{proof}
\begin{conjecture}\label{conjecture-two-filtration}
 $\Filt{q}{\pi}\subseteq \Filt{q}{\gamma}$ for all $0\leq q\leq g$.
\end{conjecture}
\begin{remark}\label{rem:conj}
 The conjecture above holds true for $q=0,1,g-1,g$. For $q=0$ there is nothing to prove. For $q=1$ note that  if $x\in {\rm ker}(\epsilon_\pi)$, then $x_g=0$ and in particular $\epsilon (x)=x^0=x^0_g=0$. Now let $q=g$. If $x\in \Filt{g}{\pi}$, then $x=x_0=x_0^g$ and hence $x\in \Filt{g}{\gamma}$. Similarly if $x\in \Filt{g-1}{\pi}$, then $x=x_0+x_1=x_0^g+x_1^{g}+x_1^{g-1}$ and hence $x\in \Filt{g-1}{\gamma}$.
\end{remark}
\begin{lemma}\label{lem:conjecture}
 Let $x\in K^p_q$ with $g-q>0$ and $p>0$. The following assertions are equivalent:
 \begin{enumerate}
  \item $p\geq g-q$,
  \item $\gamma^i_\pi(x)\in \Filt{i}{\gamma}$ for all $i\geq 0$,
  \item $\gamma^{g-q}_\pi(x)\in \Filt{g-q}{\gamma}$,
  \item $\gamma^{p+1}_\pi(x)\in \Filt{p+1}{\gamma}$.
  \end{enumerate}
\end{lemma}
\begin{proof}
Let us first show $(1)\Longrightarrow (2)$: Fix $i\geq 1$. The coefficients of $x^m$ in~\eqref{eq:gamma} is zero unless $m(g-q)\geq i$ in which case $x^m$ belongs to $\Filt{mp}{\gamma}\subseteq \Filt{i}{\gamma}$. For $(2)\Longrightarrow (3)$ there is nothing to prove. To show $(3)\Longrightarrow (1)$ note that the coefficient of $x$ in $\gamma^{g-q}_\pi (x)$ is non-zero and $(3)$ implies that $x\in \Filt{g-q}{\gamma}$ and hence $p\geq g-q$. For $(2)\Longrightarrow (4)$ there is also nothing to prove. Finally we show $(4)\Longrightarrow (1)$: if $p<g-q$, then the coefficient of $x$ in $\gamma^{p+1}_\pi(x)$ is non-zero and by $(4)$ we conclude that $x\in \Filt{p+1}{\gamma}$. This implies that $p\geq p+1$. The lemma is proved. 
\end{proof}
\begin{remark}\label{rem:conjecture-Comparison-1}
The conjecture $(F_p)$ of Beauville in~\cite[\S 5]{Beauville_1982} states that 
\[
F_pK^{(p)}\subseteq \widehat{K}^{(g-p)}\oplus \widehat{K}^{(g-p+1)}\oplus\dots\oplus \widehat{K}^{(g)}.
\]
This is equivalent to the conjecture~\ref{conjecture-two-filtration} (for $\hat{A}$). For example assume that \ref{conjecture-two-filtration} holds and let $x\in K^{(p)}$. By~\ref{lem:f_pkn} we know $y=F_p(x)\in \widehat{K}_{p}$. Using~\ref{lem:eigenspace-decomp-2} we write
\[
 y=y^0+y^1+\dots+y^g, \quad y^q\in \widehat{K}^q_p\subseteq \Filt{g-p}{\pi}.
\]
Therefore by~\ref{conjecture-two-filtration} we have $y^q\in \Filt{g-p}{\gamma}$. Hence $y^q\neq 0$ implies that $q\geq g-p$. 
\end{remark}
\section{The example of line bundles}
In this section we consider elements of ${\rm Pic}(A)$; i.e. line bundles on $A$ and do some explicit computations of the notions related to $\lambda-$ring structures introduced above when applied to line bundles.

Let $L$ be a line bundle on the fixed abelian variety $A\to \spec{k}$. We denote the $K_0-$class of $L$ by the same letter $L$. The computation related to the usual $\lambda-$ring structure on $K_0(A)$ is classical. We only mention that for any $n>0$ we have
\[
 \psi^n(L)=L^n.
\]
On the other hand similar to~\ref{prop:omega_n}, it can be shown that the maps $\pi^n(x)=n^{-g}\psi^n(x)$ are morphisms of the $\lambda-$ring $K_\star=(K_0(A)\otimes\Q,\star)$ and hence comutes with augmentation; it follows that for any line bundle $L$ and integer $n>0$ we have
\[
 \chi (L^n)=n^g\chi (L).
\]
as both are equal to $n^g\chi (\pi^n(L))$. We may also compute the Euler-Poincar\'e characteristic explicitly. For this let
\[
 L=1+l_1+l_2+\dots+l_g,\quad l_j\in K^{(j)}.
\]
be the decomposition of $L$ given by~\ref{prop:eigenspace-decomp}. Since ${\rm ch}_\gamma$ (in~\ref{GRR}) is an isomorphism of algebras over $\Q$ and ${\rm ch}_\gamma(L)=\exp (c^1(L))$ it follows that ${\rm ch}_\gamma(l_1)=c^1(L)$ and $l_j={j!}^{-1}l_1^j\in K^{(j)}$. An explicit formula for $l_1$ in $K_0(A)\otimes\Q$ may be given as:
\[
 l_1={\rm log} (L)=\sum_{n\geq 1} \frac{(-1)^{n-1}}{n}(L-1)^n.
\]
Now we consider two cases:
\begin{enumerate}
 \item If $L$ is anti-symmetric, then $n_A^\ast (L)=L^n$ for all $n\in \Z$ and hence from the formula $\exp (l_1)=L$ we obtain $l_1\in K^1_g$. It follows from~\ref{thm:fm-iso} that in general $\chi (x)[0_A]=x^g_0$ for any $x\in K_0(A)\otimes\Q$, and in particular $\chi(L)=0$. And
\[
\psi^n_\star (L)=n^gL.
\]
\item If $L$ is symmetric, then $l_j\in K^j_{g-j}$. In particular we obtain a Riemann-Roch theorem
\[
\chi (L)[0_A]=\frac{1}{g!}l_1^g=\frac{1}{g!}(L-1)^g\in K^g_0.
\]
And since $\pi_\star^n$ is by~\ref{prop:omega_n} a ring homomorphism we conclude that
\begin{alignat*}{2}
 \psi^n_\star(L) & =n^g\pi^n(L)\\
 & = n^g\exp (\pi^n(l_1))\\
 & =n^g\exp (n^{-1}l_1)\\
 & =:n^g\sqrt[n]{L}.
\end{alignat*}
\end{enumerate}
In particular and as expected the Pontryagin $\lambda_\star-$structure (cf. before~\ref{thm:fm-iso}) does see the difference btween symmetric and anti-symmetric line bundles but the usual $\lambda-$ring structure does not. 
\section{Composition of two $\lambda-$ring structures}
Now we consider certain composition of the usual (Grothendieck) $\lambda-$ring structure with the Pontryagin $\lambda_\star-$structure on $K=K_0(A)\otimes\Q$. This shall give another equivalent form of Beauville's Conjecture.

Keeping the notations of the previous section, define for each $n\geq 1$ and $x\in K$:
\[
 \Psi^n(x)=n^{-g}\psi_\star^n\circ \psi^n(x)
\]
for $x\in K$. It follows from~\ref{prop:omega_n} that $\Psi^n$'s are $\Q-$algebra homomorphism on $K_0(A)\otimes\Q$ and have the properties: $\Psi^n\circ \Psi^m=\Psi^{nm}$, and $\Psi^1=\id{}$. Therefore by [SGA 6, Exp. V, Prop. 7.5], there is a $\lambda-$ring structure $\Lambda$ on $K_0(A)\otimes\Q$ whose corresponding Adams operations are $\Psi^n$.
\begin{example}\label{ex:Lambda-line-bundle}
Let $L$ be a line bundle on $A$. It follows that for each $n\geq 1$ we have $\Psi^n(L)=L^n$ if $L\in {\rm Pic}^0(A)$ and $\Psi^n(L)=L$ if $L$ is symmetric.
\end{example}

For each $-g\leq j\leq g$ define the sub-space $K[j]\subseteq K_0(A)\otimes\Q$ to be generated by $K^p_q$ with $p+q-g=j$. By~\ref{lem:eigenspace-decomp-2} and~\ref{prop:eigenspace-decomp}, the space $K$ is a direct sum of $K[j]$'s. The subspace $K[0]$ is in fact a $\Q-$subalgebra. Let
\[
 \epsilon_\Gamma\colon K_0(A)\otimes\Q\to K[0]
\]
be the natural projection map. The ring $K[0]$ has the induced structure of a binomial $\lambda-$ring\footnote{Comared to the definition in [SGA 6, Exp. V, 2.7], it seems better to define a binomial ring to be any $\Z-$algebra with a $\lambda-$ring structure for which the Adams operations are all the identity maps. Over the rationals these are of course equivalent.};
\[
 \lambda_\beta(x,t)=\exp(x\ln (1+t)).
\]
\begin{conjecture}\label{conjecture-2}
$K[i]\cdot K[-j]=0$ for all $i,j>0$.
\end{conjecture}
\begin{remark}\label{rem:conjecture}
The conjecture above is implied by the conjecture $(F_p)$ of Beauville in~\cite{Beauville_1982} which can be stated as $K[-j]=0$ for all $j>0$. On the other hand (the first part of) the conjecture of Bloch~\cite[0.2]{Bloch} can be stated as $K^r_g\cdot K^s_n=0$ for all $n\geq 0$, $r\geq n+1$ and $s\geq 0$. This is implied by~\ref{conjecture-2}; first note that $K^r_g\subseteq K[r]$ and $K^s_n\subseteq K[s+n-g]$. Now if $s+r-g>0$, then $K^r_g\cdot K^s_n\subseteq K^{r+s}_{n}=0$ by definition. Otherwise $s+n-g<s+r-g\leq 0$ and hence $K^r_g\cdot K^s_n=0$ by~\ref{conjecture-2}.  
\end{remark}
From now on and unless otherwise stated we assume that the conjecture~\ref{conjecture-2} holds.\\

Note that the algebra $K_0(A)\otimes\Q$ becomes a $K[0]-$augmented $\lambda-$ring:
\begin{lemma}
 The map $\epsilon_\Gamma\colon K_0(A)\otimes\Q\to K[0]$ is a $\lambda-$morphism of $K[0]-$algebras.
\end{lemma}
\begin{proof}
 Since $K[j]\cdot K[-j]=0$ for all $j>0$, it follows that $\epsilon_\Gamma$ is a homomorphism of $K[0]-$algebras. Next we need to show that for any $x\in K_0(A)\otimes\Q$ and $i\geq 0$, we have
 \[
  \epsilon_\Gamma \Lambda^i(x)=\lambda_\beta^i (\epsilon_\Gamma(x)).
 \]
This follows from the fact that $\epsilon_\Gamma$ is a $\Q-$algebra homomorphism and $\epsilon_\Gamma\circ \Psi^n=\epsilon_\Gamma$. 
\end{proof}
\begin{example}
Let $L$ be a symmetric line bundle on $A$. It follows from~\ref{ex:Lambda-line-bundle} that (the $K_0-$class of) $L$ belongs to $K[0]$.
\end{example}

Let $\Filt{r}{\Gamma}$ be the corresponding gamma filtration, e.g. $\Filt{0}{\Gamma}=K_0(A)\otimes\Q$ and $\Filt{1}{\Gamma}={\rm ker (\epsilon_\Gamma})$. Let 
\[
 \grr{r}{\Gamma}{K_0(A)\otimes\Q}:=\Filt{r}{\Gamma}/{\Filt{r+1}{\Gamma}}
\]
be the corresponding quotient $K[0]-$modules. For any $x\in K_0(A)\otimes\Q$ and $i\geq 1$, denote by $c^i_\Gamma(x)\in \grr{i}{\Gamma}{K_0(A)\otimes\Q}$ the $i-$th Chern class of $x$. The Chern ring $${\rm Ch}_{\Gamma}(K_0(A)\otimes\Q):=K[0]\times (1+\prod_{r>0}\grr{r}{\Gamma}{K_0(A)\otimes\Q})$$ is defined and is a $K[0]-$augmented $\lambda-$ring. There is a well-defined complete Chern class morphism
 \[
  \tilde{c}_\Gamma\colon K_0(A)\otimes\Q\to {\rm Ch}_\Gamma(K_0(A)\otimes\Q),\quad x\mapsto (\epsilon_\Gamma (x), 1+c^1_\Gamma(x)+c^2_\Gamma(x)+\dots)
 \]
By [SGA 6, Exp. V, 6.8] this is a $\lambda-$morphism of $K[0]-$augmented $\lambda-$rings.
\begin{lemma}\label{lem:Fil1}
$\Filt{1}{\Gamma}\subseteq \Filt{1}{\gamma}$. 
\end{lemma}
\begin{proof}
 This is clear.
\end{proof}
Let $x\in K_0(A)\otimes\Q$. Using the decompositions~\ref{prop:eigenspace-decomp} and~\ref{lem:eigenspace-decomp-2}, we may uniquely write
\[
 x=x[-g]+\dots+x[0]+x[1]+\dots
\]
where $x[j]\in K[j]$ for all $-g\leq j\leq g$. Let $r\geq 0$ and note that:
\begin{lemma}\label{lem:Fil2}
 $K[j]\subseteq \Filt{r}{\Gamma}$ if $j<0$ or $j\geq r$.
\end{lemma}
\begin{proof}
 This is left to the reader.
\end{proof}
\begin{proposition}\label{prop:kernel-c}
${\rm ker}(\tilde{c}_\Gamma)=\Filt{g+1}{\Gamma}$.
\end{proposition}
\begin{proof}
Let us first show that \[
 {\rm ker}(\tilde{c}_\Gamma)=\bigcap_{r\geq 0}\Filt{r}{\Gamma}.
\]
This is easy (and generally true): if $\tilde{c}_\Gamma(x)=0$, then $\epsilon_\Gamma(x)=0$ and $c^i(x)=0$ for all $i\geq 1$. For $i=1$, this gives $x\in \Filt{2}{\Gamma}$. Using an induction on $i$ and the result [SGA 6, Exp. V, Prop. 6.9], it follows that $x\in \Filt{i}{\Gamma}$ for all $i\geq 0$. Next we show that for each $0\leq r\leq g$ we have $\Filt{r}{\Gamma}\neq \Filt{r+1}{\Gamma}$. To see this assume $r\geq 1$. Using the powers of elements of ${\rm Pic}^{0}(A)$, we know that $K^r_g\neq 0$ for all $1\leq r\leq g$ and note that $K^r_g\subseteq K[r]\subseteq \Filt{r}{\Gamma}$. Let $0\neq x\in K^r_g\cap \Filt{r+1}{\Gamma}$. Therefore we can write
\[
 x=\sum a\Gamma^i(x)\cdot \Gamma^j(y)\cdot\ldots\cdot \Gamma^k(z)
\]
where $a\in K[0]$, $i+j+\dots+k\geq r+1$ and $x,y,\dots,z\in \Filt{1}{\Gamma}$. Using the decompositions~\ref{rem:eigenspace-decomp} and~\ref{lem:eigenspace-decomp-2}, we may assume that $a, x,y,\dots,z$ are homogeneous, i.e. each one is an element of some $K^p_q$. The uniqueness of the decompositions and the property
\[
 K^a_b\cdot K^\alpha_\beta\subseteq K^{a+\alpha}_{b+\beta-g}
\]
imply that we may assume that all $x,y,\dots,z$ belong to $K_g$ and $a\in \Q$. In this case $\Gamma^i(x)=\gamma^i(x)$ and similarly for other values of $i$ and $x$. Therefore $x\in \Filt{r+1}{\gamma}$. This can not happen because ${\rm ch}_\gamma(x)\neq 0$. Finally we show that $\Filt{g+1}{\Gamma}=\Filt{g+2}{\Gamma}=\dots$. For this let $x\in \Filt{r}{\Gamma}\setminus \Filt{r+1}{\Gamma}$ for some $r>0$. By the first part of the proof and the fact from~\ref{lem:Fil2} that $K[-j]\subseteq {\rm ker}(\tilde{c}_\Gamma)$ for all $j>0$, it follows that $0\neq x[j]\in \Filt{r}{\Gamma}\setminus \Filt{r+1}{\Gamma}$ for some $j>0$. Using the fact that $\Psi^n(x[j])=n^jx$ for all $n\geq 1$ and $\Psi^n(x[j])-n^rx[j]\in \Filt{r+1}{\Gamma}$ (by the general properties and definitions of any $\gamma-$filtration), it follows that $j=r$, i.e. $r\leq g$. The result follows. 
\end{proof}

\begin{conjecture}\label{conjecture-3}
$\Filt{g+1}{\Gamma}=0$.
\end{conjecture}
\begin{remark}
This conjecture is equivalent to the conjecture $(F_p)$ of Beauville (cf.~\ref{rem:conjecture-Comparison-1}). To see this note that~\ref{conjecture-3} implies that the complete Chern class $\tilde{c}_\Gamma$ is injective. Note that if $0\neq x\in K^p_q$ with $p,g-q>0$, then $x\in K[j]$ with $j=p+q-g$ and hence for $j<0$ we would have $x\in {\rm ker}(\tilde{c}_\Gamma)$. Therefore $p+q-g\geq 0$. On the other hand the conjecture of Beauville implies that $K[-j]=0$ for all $j>0$. Therefore every $0\neq x\in K_0(A)\otimes\Q$ may be written as
\[
 x=x[0]+x[1]+\dots+x[g]
\]
where $x[j]\in K[j]$ for all $0\leq j\leq g$. Using (the proof of) the result~\ref{prop:kernel-c} we conclude that $x\not\in {\rm ker}(\tilde{c}_\Gamma)$.
\end{remark}
\begin{example}
 Let $x\in K[-1]$. It follows that $\tilde{c}_\Gamma(x)=0$ and $x\in \Filt{r}{\Gamma}$ for all $r\geq 0$. We may write
\[
 \Lambda_t(x)=\exp \bigl(x\int \frac{\ln(1+t)}{t}{\rm d}t\bigr)
\]
understood as formal identities in the ring of power series in $K_0(A)\otimes\Q$. Similarly we may write
\[
 \Gamma_t(x)=\exp \bigl(xt+\frac{3x}{2!}t^2+\frac{11x}{3!}t^3+\frac{50x}{4!}t^4+\frac{274x}{5!}t^5+\dots\bigr)
\]
where the numbers in the numerators of the coefficients are certain Stirling numbers.\footnote{The conjectures above seem (imprecisely speaking) to be stating that exponentials of the above type may not be (algebraically) defined in a ring such as $K_0(A)\otimes\Q$ with available exponentials of the form
\[
 \lambda_t(x)=\exp \bigl(x\sum (-1)^{n-1}n^{q-1}t^n)
\]
for $q\geq 0$.}
\end{example}

\end{document}